\documentclass[preprint]{imsart}

\RequirePackage[OT1]{fontenc}
\RequirePackage{amsthm,amsmath}
\RequirePackage[numbers]{natbib}
\RequirePackage[colorlinks,citecolor=blue,urlcolor=blue]{hyperref}
\usepackage{caption}
\usepackage{amsmath,amssymb,amsthm}
\usepackage{amsmath}
\usepackage{amsfonts}
\usepackage{float}
\usepackage{mathtools}
\usepackage{mathrsfs}  
\usepackage{youngtab}
\usepackage[pdftex]{graphicx}
\usepackage{tikz}
\usetikzlibrary{decorations.markings,arrows,automata,positioning}
\tikzset{partition/.style={fill,circle,inner sep=1pt}}
\usetikzlibrary{positioning}
\usetikzlibrary{decorations.pathreplacing}

\usetikzlibrary{matrix,arrows}
\usetikzlibrary{calc}
\tikzset{partition/.style={fill,circle,inner sep=1pt},
         part/.style={baseline=0,scale=0.5,bend left=45},
         partlabel/.style={below}}
\usetikzlibrary{shapes,arrows}
\usetikzlibrary{snakes}
\tikzstyle{pnt}=[draw,ellipse,fill,inner sep=1pt]

\tikzstyle{opnt}=[draw,ellipse,inner sep=1pt]
\tikzstyle{opnt}=[ ]
\tikzstyle{pntt}=[draw,ellipse,fill,inner sep=0.5pt]
\tikzstyle{point}=[draw,ellipse,fill,inner sep=2pt]

\usepackage{color}

\DeclareMathOperator{\Tr}{Tr}

\usepackage{amssymb}

\theoremstyle{plain}
\newtheorem{theorem}{Theorem}[section]
  \newtheorem{lemma}[theorem]{Lemma}

  \newtheorem{proposition}[theorem]{Proposition}
  \newtheorem{corollary}[theorem]{Corollary}
  	\newtheorem{definition}[theorem]{Definition} 
  	\theoremstyle{definition}
  	\newtheorem{example}[theorem]{Example} 
  \theoremstyle{remark}  	
  	
  	\newtheorem{remark}[theorem]{Remark}
	\newtheorem{note}[theorem]{Note}
\allowdisplaybreaks

\newcommand{\R}{\mathbb{R}}

\begin{document}

\begin{frontmatter}
\title{Random Walks on the BMW Monoid: an Algebraic Approach}
\runtitle{Random Walks on the BMW Monoid}

\begin{aug}
\author{\fnms{Sarah} \snm{Wolff}\thanksref{t1}\ead[label=e1]{wolffs@denison.edu}}

\thankstext{t1}{Partially supported by the NSF GRFP under Grant No. DGE-1313911}
\runauthor{S. Wolff}

\affiliation{Denison University }

\address{Department of Mathematics and Computer Science\\
Denison University\\ 
\printead{e1}\\
\phantom{E-mail:\ }}

\end{aug}

\begin{abstract}

We consider Metropolis-based systematic scan algorithms for generating Birman-Murakami-Wenzl (BMW) monoid basis elements of the BMW algebra. As the BMW monoid consists of tangle diagrams, these scanning strategies can be rephrased as random walks on links and tangles. We translate these walks into left multiplication operators in the corresponding BMW algebra. Taking this algebraic perspective enables the use of tools from representation theory to analyze the walks; in particular, we develop a norm arising from a trace function on the BMW algebra to analyze the time to stationarity of the walks. 
\end{abstract}



\end{frontmatter}


\section{Introduction} 

Studying the convergence of random walks on finite groups, and in particular the problem of generating group elements according to a fixed probability distribution has a long history \cite{scarabotti, diaconis, diacoste2, saloff}. Of particular interest for the purposes of this paper is the important work of Diaconis and Ram \cite{diaram}, who compare systematic scanning techniques with random scanning techniques in the context of generating elements of a finite Coxeter group $W$ using the Metropolis algorithm. 

First introduced by Metropolis, Rosenbluth, Rosenbluth, Teller, and Teller \cite{met}, the Metropolis algorithm gives a method for sampling from a probability distribution $\pi$ by modifying an existing Markov chain to produce a new chain with stationary distribution $\pi$. This proves particularly useful for simulating configurations of particles with an associated energy (e.g., the influence that neighboring particles exert on each other). Later applications of the Metropolis algorithm include the simulation of Ising models, initially developed to model a ferromagnet but (surprisingly) also of use in image analysis and Gibbs sampling \cite{cai, fishman}. See \cite{liu} for additional applications.
The Metropolis algorithm has the advantage of being straightforward to construct and implement; however, in analyzing the rate of convergence to $\pi$ (the \textit{mixing time})  rigorous bounds are often dependent on the specific situation (see \cite{pedthesis} for a review of the existing literature for spin systems alone).  Further, these methods are most often examples of \textit{random scan Markov chains} in that the process involved is that of selecting a site or set of sites to update at random. 
%
A more intuitively appealing and often more frequently used method in experimental work is that of a \textit{systematic scan Markov chain}: a method to cycle through and update the sites in a deterministic order.
While such scanning strategies may seem intuitive for use in sampling from $\pi$, they have proven difficult to analyze in many situations.
%

In \cite{diaram} Diaconis and Ram use the Metropolis algorithm construction to produce Markov chains $M_1,M_2,\dots,$ $M_{n-1}$ corresponding to multiplication by the generators $r_1,\cdots, r_{n-1}$ of a Coxeter group $W$. 
These Markov chains provide systematic scanning strategies for multiplying by generators of $W$ (for an explicit description of $M_i$ and the corresponding random walk see Section \ref{thewalk}). Diaconis and Ram \cite{diaram} show that convergence of the short systematic scan occurs in the same number of steps as that of a random scan.




The key insight that allows for analysis of the Metropolis scans is the translation of the Markov chains $M_i$ into left multiplication operators in the \textit{Iwahori-Hecke Algebra} corresponding to $W$. 
Hecke algebras arise naturally in the extension of Schur-Weyl duality to general centralizer algebras. More relevant for this paper is an alternative definition of the Hecke algebra in terms of \textit{braids}. The thesis \cite{simplylaced} gives 
a thorough introduction to braids and their relationship with the Hecke algebra. 

Let $b_1,\dots, b_n\in\R$ with $b_1<\cdots <b_n$. An $n$-strand braid is a disjoint union of $n$ smooth curves in $\R^3$ connecting the points $\{(b_1,1,0),(b_2,1,0),\dots, (b_n,1,0)\}$ with $\{(b_1,0,0),(b_2,0,0),\dots, (b_n,0,0)\}$ so that they intersect each parallel plane $y=t$ as $t$ ranges between $0$ and $1$ only once. A braid can be represented by its 2-dimensional projection, its \textit{braid diagram}, and connecting the top strands to the bottom strands of a braid diagram gives rise to a \textit{link}. Two links are \textit{isotopic} if they are related by a sequence of Reidemeister moves (defined in Section \ref{BMWdefs}), and, in fact, every isotopic oriented link can be represented by the closure of a braid \cite{simplylaced}. The braid group has a presentation in terms of generators $T_{r_1},\dots,  T_{r_{n-1}}$ corresponding to certain braid diagrams. Remarkably, adding a quadratic relation to this presentation yields the Hecke algebra.

Under this definition of the Hecke algebra there is a natural generalization to the \textit{Birman-Murakami-Wenzl (BMW) algebra}. 
By now allowing any two points in $\{(b_1,1,0),(b_2,1,0),\dots, (b_n,1,0)\}\cup\{(b_1,0,0),(b_2,0,0),\dots, (b_n,0,0)\}$  to be connected, we have the definition of an \textit{n-tangle}, which gives rise to the idea of a \textit{tangle diagram} by considering its two-dimensional projection. We define tangle diagrams in detail in Section \ref{BMWdefs}. As with the algebra associated to braid diagrams, an algebra is associated to these tangle diagrams. Defined independently as the \textit{Kauffman tangle algebra} by Murakami \cite{murakami} and algebraically by Birman and Wenzl \cite{birmanwenzl}, it was shown in an unpublished paper by Wasserman \cite{wasser} that these two notions are equivalent, giving rise to the single BMW algebra. 



In \cite{diaram}, Diaconis and Ram consider the problem of systematically generating elements of a finite Coxeter group $W$. In terms of the group algebra $\mathbb{C}[W]$, this problem is equivalent to generating elements of the basis $W$ of $\mathbb{C}[W]$. We extend these ideas to the BMW algebra. The Metropolis algorithm in this context gives rise to systematic scanning strategies for generating basis elements via multiplication of generators.
As the diagrams forming the BMW monoid basis of the BMW algebra are tangles, scanning strategies for generating BMW monoid elements have applications arising in physics: random generation of links and tangles has been of use in \cite{rangenknots, componentlinks, enumlink}. As in \cite{diaram}, our algorithm gives rise to a natural random walk, in this case on the BMW and Brauer monoids, defined in Section \ref{thewalk}. We translate the random walk into multiplication in the BMW algebra: for $\mathscr{T}_{r_i}, \mathscr{T}_{e_i}$ left multiplication operators in the BMW algebra.

\begin{theorem}\label{leftmult1} The chain $K_i$ arising from the Metropolis algorithm is the same as the matrix of left multiplication by 
$$\theta\mathscr{T}_{r_i}+(1-\theta)\mathscr{T}_{e_i}.$$
\end{theorem} 

The main tool used in the analysis in \cite{diaram} is Proposition 4.6, which translates the total variation norm into an inner product on the Iwahori-Hecke algebra $H$ arising from a trace on $H$. Plancherel's theorem then allows for bounds using the dimensions and characters of representations of $H$.

We extend the natural trace function on the Hecke algebra to the BMW algebra to provide an analogue of Proposition 4.6 (Theorem \ref{main2}). We develop a trace form $\langle,\rangle_{BMW}$ to study the walk, similarly enabling the use of tools from representation theory to analyze the time to stationarity of such walks. We consider submatrices $\hat{K}$ of $K_i$ with respect to a shifted basis. Let $\hat{\pi}$ denote the stationary distribution of $\hat{K}$.

\begin{theorem}\label{main2}
$$\| [\hat{K}^n/\hat{\pi}]_{x}-1\|^2_2 \leq \|[\hat{K}^n]_{x}-1\|_{BMW}^2.$$
\end{theorem}
Thus, studying the time to stationarity of $\hat{K}$ can be achieved by studying $\|[\hat{K}^n]_{x}-1\|_{BMW}^2$. This opens up representation theoretic tools---in particular the dimensions and traces of representations of the BMW algebra---for studying the random walk. 

We begin in Sections \ref{probprelim} and \ref{prelim2} with the preliminaries needed from the probability theory and the representation theory of semisimple algebras. We also give a presentation of the Brauer and BMW algebras. In Section \ref{thewalk} we describe the random walk arising from the Metropolis algorithm, and prove Theorem \ref{leftmult1}. We continue in Section \ref{walkanalysis} with analysis of the walk, recasting it in terms of a translated basis, constructing a trace form to bound the time to stationarity, and proving Theorem \ref{main2}.

\section{Preliminaries: Probability Theory}\label{probprelim}

Background on Markov chains can be found in many standard probability texts (see eg \cite{feller}). The book of Levin, Peres, and Wilmer \cite{levinetall} gives a particularly thorough introduction to Markov chains, including classification of states and the Metropolis algorithm, while \cite{diaram} gives a concise introduction to the probabilistic background needed. We will follow the notation and outline of \cite{diaram}.

\subsection{Markov Chains}\label{L2norm}
A finite Markov chain with state space $X$ is a process that moves among states in $X$ such that the conditional probability of moving from state $x$ to state $y$ is independent of the preceding sequence of states. More formally:

\begin{definition}
A \textbf{Markov chain} on a finite set $X$ is a matrix $K=(K(x,y))_{x,y\in X}$ such that $K(x,y)\in[0,1]$ and for all $x\in X$, $$\sum_{y\in X} K(x,y)=1.$$ We call $X$ the \textbf{state space}.
\end{definition}
Note that $K(x,y)$ gives the probability of moving from $x$ to $y$ in one step, while $K^m(x,y)$ gives the probability of moving from $x$ to $y$ in $m$ steps. 
\begin{definition} A Markov chain $K$ is \textbf{irreducible} if for each $x,y\in X$, there exists an integer $m$ such that $K^m(x,y)>0$. Let $T(x)$ denote the minimum $t$ such that $K^t(x,x)>0$. Then $K$ is \textbf{aperiodic} if $$\gcd_x(T(x))=1.$$ 
\end{definition}
Note that if $K$ is irreducible and aperiodic, there exists an integer $r$ such that $K^r(x,y)>0$ for all $x,y\in X$ \cite[Proposition 1.7]{levinetall}. 
\begin{definition} A Markov chain is \textbf{reversible} if there exists a probability distribution $\pi:X\rightarrow [0,1]$ such that for all $x,y\in X$,
$$\pi(x)K(x,y)=\pi(y)K(y,x).$$ We call $\pi$ the \textbf{stationary distribution} of $K$.
\end{definition}

An irreducible, aperiodic, reversible Markov chain $K$ converges to its stationary distribution: 
$$\lim_{m\rightarrow\infty} K^m(x,y)=\pi(y).$$ 

The Metropolis construction introduced in Section \ref{met} produces a reversible Markov chain with a chosen stationary distribution. Our interest is in the time to stationarity of such chains. 

\begin{definition}
Let $K^m_x$ denote the probability distribution $K^m(x,\cdot)$. The \textbf{total variation distance} from $K^m_x$ to $\pi$ is 
$$\vert K^m_x-\pi\vert_{TV}:=\max_{A\subseteq X}\vert \sum_{y\in A} K^m(x,y)-\pi(y)\vert.$$
\end{definition}
For $L^2(\pi)$ the space of functions $f:X\rightarrow\mathbb{R}$, equipped with the inner product
$$\langle f,g\rangle_2=\sum f(x)g(x)\pi(x),$$ the total variation distance is bounded by the $L^2(\pi)$ norm:

\begin{lemma}\label{tvl2}\cite[Lemma 2.3]{diaram}
For $f\in L^2(\pi)$, $$|f|_{TV}^2\leq \frac{1}{4}\|f/\pi\|_2^2,$$ where $f/\pi(x)=0$ if $\pi(x)=0$.
\end{lemma}

\subsection{The Metropolis Algorithm}\label{met}

Given a symmetric Markov chain $P$ and a probability distribution $\pi$, the Metropolis algorithm modifies $P$ to produce a reversible Markov chain $M$ with stationary distribution $\pi$:

$$M(x,y)=\left\{
\begin{array}{ll} 
\displaystyle P(x,y)& \text{if } x\neq y \text{ and } \pi(y)\geq \pi(x),\\
\displaystyle P(x,y)\frac{\pi(y)}{\pi(x)}& \text{if } x\neq y \text{ and } \pi(y)<\pi(x),\\
\displaystyle P(x,x)+\sum_{\pi(z)<\pi(x)} P(x,z)\left(1-\frac{\pi(z)}{\pi(x)}\right)& \text{if } x=y.
\end{array}\right.$$


While $M(x,y)$ is reversible with stationary distribution $\pi$, irreducibility and aperiodicity are not guaranteed. In particular, the Markov chains we consider in Section \ref{thewalk} are aperiodic but not irreducible. To analyze these chains we consider their \textit{closed communication classes}.

\begin{definition}
Let $K$ be a Markov chain with state space $X$. For $x,y\in X$, $y$ \textbf{is accessible} from $x$, denoted $x\rightarrow y$, if $x$ can reach $y$ in finitely many steps. We say $x$ \textbf{communicates with} $y$, denoted $x\leftrightarrow y$, if $x\rightarrow y$ and $y\rightarrow x$. The equivalence classes under the relation $\leftrightarrow $ are the \textbf{communication classes} of $K$. A communication class $C$ is \textbf{closed} if for $x\in C$ and for all $y\notin C$, $y$ is not accessible from $x$. 
\end{definition}
Note that studying the time to stationarity of a reversible, aperiodic Markov chain $K$ reduces to studying the time to stationarity of the closed communication classes of $K$.

\subsection{Systematic Scans}
The Metropolis algorithm, in the context of generating elements of a group, provides systematic and random scanning strategies. For example, for each generator $r_i=(i\;i+1)$ of $S_n$, let
$$   P_i(x,y) = \left\{
     \begin{array}{ll}
       1 &  \text{if}\;y=r_ix,\\
       0 &  \text{else}.
     \end{array}
   \right.
$$
Then for $l_{S}$ the length function on words in $S_n$, let $\pi$ be the probability distribution $$\pi(x)=\frac{\theta^{-l_{S}(x)}}{\displaystyle\sum_{w\in S_n} \theta^{-l_{S}(w)}}.$$ 
The Metropolis algorithm construction then produces Markov chains $M_1,M_2,\dots,$ $M_{n-1}$ corresponding to multiplication by the generators $r_1,\cdots, r_{n-1}$. For an explicit description see Section \ref{thewalk}. 

A choice of infinite sequence $\{i_l\}_{l=1}^\infty$ gives a scanning strategy: $$\cdots M_{i_l}M_{i_{l-1}}\cdots M_{i_1}.$$ For $M_i$ reversible, each with stationary distribution $\pi$, the following systematic scans produce reversible Markov chains with stationary distribution $\pi$ (see, eg \cite{diaram}):
$$\begin{array}{ll}
      \displaystyle \frac{1}{n-1}\sum_{i=1}^{n-1} M_i & \text{(random scan)},\\
       M_1M_2\cdots M_{n-1}M_{n-1}\cdots M_2M_1 & \text{(short systematic scan)},\\
       (M_1\cdots M_{n-1}M_{n-1}\cdots M_1)\cdots (M_1M_2M_2M_1)(M_1M_1) & \text{(long systematic scan)}.\\
     \end{array}
$$ 

While such scanning strategies may seem intuitive for sampling from $\pi$, they have proven difficult to analyze in many situations. In the context of generation of Coxeter group elements,  Diaconis and Ram \cite{diaram} show that convergence of the short systematic scan for the distribution $\pi$ above,  with $l_S$ replaced by the length function on the Coxeter group coming from writing words as a product of simple reflections, occurs in the same number of steps as that of a random scan, i.e., choosing a random sequence of indices $\{i_\ell\}_{\ell=1}^\infty$. However, results for different scanning techniques or probability distributions remain open. In the context of graph colorings, Dyer et al. compare systematic scans with random scans for sampling proper $q$-colorings of paths for $q\geq 4$, in which a vertex is assigned a new color $c$ only if none of its neighbors are colored by $c$ \cite{dyeretall}. However, results for more general graphs have resisted analysis. 

Fishman \cite{fishman} gives an overview of scanning strategies, while Diaconis and Saloff-Coste's survey \cite{diacoste} provides further applications of the Metropolis algorithm.

\section{Preliminaries: Semisimple Algebras}\label{prelim2}
\subsection{Fourier Inversion and Plancherel}\label{planchinv}
Random walks on groups are frequently studied using Fourier analysis. For example, for a group $G$ and a function $Q:G\rightarrow \mathbb{C}$, let $\hat{Q}$ denote the Fourier transform of $Q$.  
\begin{theorem}[Diaconis, \cite{diaconis}]\label{diacthm} For $G$ a group, $Q$ a probability distribution on $G$, and $U$ the uniform distribution on $G$, 
$$\vert Q-U\vert_{TV}^2\leq\frac{1}{4}\sum_\rho d_\rho \Tr(\hat{Q}(\rho)\hat{Q}(\rho)^*),$$
where $*$ denotes conjugate transpose and the sum is over all nontrivial irreducible representations $\rho$ of $G$.
\end{theorem}
The Fourier transform of a complex valued function on a finite group arises as a special case of Fourier transforms on semisimple algebras. Here we review the basic concepts and definitions. For more background on the representation theory of semisimple algebras see \cite{ram}. 

\begin{definition} A \textbf{matrix representation} of a $\mathbb{C}$-algebra $A$ is an algebra homomorphism $$\rho: A\rightarrow M_d(\mathbb{C}),$$ where $M_{d}(\mathbb{C})$ denotes the complex algebra of $d\times d$ matrices with entries in $\mathbb{C}$. We call $d$ the \textbf{dimension} of $\rho$.

An algebra $A$ is \textbf{simple} if $A\cong M_n(\mathbb{C})$ for some $n\geq 1$ and \textbf{semisimple} if it decomposes as a direct sum of simple algebras: $$A\cong \bigoplus_{\lambda\in \Lambda}M_{\lambda}(\mathbb{C}),$$ for a finite index set $\Lambda$. 
\end{definition}

\begin{definition}\label{alg} 
Let $A$ be a semisimple algebra, $\{a_i\}_{i\in I}$ a basis for $A$ and $\displaystyle f=\sum_{i\in I}f(a_i)a_i\in A$.
\begin{itemize}
\item[(i)] Let $\rho$ be a matrix representation of $A$. Then the \textbf{Fourier transform of} $f$ \textbf{at} $\rho$, denoted $\hat{f}(\rho)$, is the matrix sum
$$\hat{f}(\rho)=\sum_{i\in I} f(a_i)\rho(a_i).$$

\end{itemize}
\end{definition}

\begin{definition} For $A$ a semisimple algebra, a \textbf{trace} function on $A$ is a $\mathbb{C}$-linear function $\tau:A\rightarrow\mathbb{C}$ such that for all $a,b\in A$,
$$\tau(ab)=\tau(ba).$$
\end{definition}

Note by linearity that the usual trace function on $M_d(\mathbb{C})$ is unique up to multiplication by a constant. Hence, for any trace $\tau$ on $A$ and set $R$ of inequivalent irreducible representations of $A$, there exist constants $t_\rho\in\mathbb{C}$ such that:
$$\tau=\sum_{\rho\in R}t_\rho T_\rho,$$
where for $a\in A$, $T_\rho(a)=\Tr(\rho(a))$.

A trace function $\tau$ gives rise to a symmetric bilinear form $\langle \cdot,\cdot\rangle_\tau:A\times A\rightarrow \mathbb{C}$ by letting $$\langle a,b\rangle_\tau=\tau(ab),$$
for $a,b\in A$.

Both Theorem \ref{diacthm} and the results of \cite{diaram} require the notion of Fourier inversion and Plancherel's Theorem.
\begin{theorem}[Fourier Inversion, Plancherel]\label{planch}
Let $A$ be a semisimple algebra with basis $\{a_i\}$ and $\tau$ a nondegenerate trace on $A$.  Let $\{a_i^*\}$ be the dual basis to $\{a_i\}$ with respect to the trace form $\langle \cdot,\cdot\rangle_\tau$. Then for $f, f_1, f_2$ complex-valued functions on $A$,
\begin{equation}
f(a_i)=\sum_{\rho} t_\rho \Tr(\hat{f}(\rho)\rho(a_i^*)),
\end{equation}

\begin{equation}
\langle f_1, f_2\rangle_\tau=\sum_{\rho} t_\rho \Tr(\hat{f_1}(\rho)\hat{f_2}(\rho)).
\end{equation}

\end{theorem}

\subsection{The Brauer Algebra}\label{brauerdefs2}

Elements of the \textbf{Brauer monoid}, $Br_n$, are realized as generalized symmetric group diagrams: consider diagrams on $2$ rows of $n$ points each, with edges connecting pairs of points regardless of row and each point part of exactly one edge. Multiplication is realized as concatenation of diagrams. Note that in some cases, concatenation introduces a closed loop. For a parameter $q$ and two diagrams $x,y\in Br_n$, let $c$ denote the number of closed loops in the multiplication $xy$ and let $z$ be the diagram of this product with the closed loops removed. Then $xy=q^c z$.  

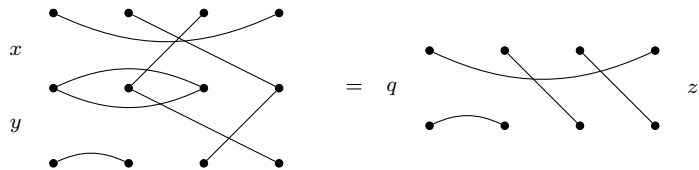
\begin{figure}[H] 
\begin{center}
\begin{tikzpicture}
   \node[pnt] at (-3,0) (v_1) {};
   \node[pnt] (v_2) [right of=v_1] {};
   \node[pnt] (v_3) [right of=v_2] {};
   \node[pnt] (v_4) [right of=v_3] {};
   \node[pnt] (v_5) [below of=v_1] {};
   \node[pnt] (v_6) [below of=v_2] {};
   \node[pnt] (v_7) [below of=v_3] {};
   \node[pnt] (v_8) [below of=v_4] {};
   \node[pnt] (w_5) [below of=v_5] {};
   \node[pnt] (w_6) [below of=v_6] {};
   \node[pnt] (w_7) [below of=v_7] {};
   \node[pnt] (w_8) [below of=v_8] {};
   \node (q) [right of=v_8] {=};
   \node at (-3.5,-.5) (x) {$x$};
   \node at (-3.5,-1.5) (x) {$y$};
   \node (q) at (1.5,-1) {$q$};
   \node[pnt] at (2,-.5) (u_1) {};
   \node[pnt] (u_2) [right of=u_1] {};
   \node[pnt] (u_3) [right of=u_2] {};
   \node[pnt] (u_4) [right of=u_3] {};
   \node[pnt] (u_5) [below of=u_1] {};
   \node[pnt] (u_6) [below of=u_2] {};
   \node[pnt] (u_7) [below of=u_3] {};
   \node[pnt] (u_8) [below of=u_4] {};
   \node at (5.5,-1) (x) {$z$};
   \path[-]
    (v_1) edge [bend right=25] node {} (v_4)
    (v_2) edge  node {} (v_8)
    (v_3) edge  node {} (v_6)
    (v_7) edge [bend right=25] node {} (v_5)
    (v_7) edge [bend left=25] node {} (v_5)
    (v_6) edge node {} (w_8)
    (w_6) edge [bend right=25]   node {} (w_5)
    (v_8) edge node {} (w_7)   
    (u_1) edge [bend right=25]  node {} (u_4)
    (u_2) edge  node {} (u_7)
    (u_3) edge  node {} (u_8)
    (u_5) edge [bend left=25]  node {} (u_6);    

\end{tikzpicture}
\caption{ $xy=q^1z$}\label{brauers1}
\end{center}
\end{figure}

Two Brauer diagrams $d_1$ and $d_2$ are equivalent if they differ only in the number of closed loops, i.e., if when $q=1$, $d_1=d_2$. For example, for $x,y,z$ as in Figure \ref{brauers1}, the product $xy$ is equivalent to $z$. The Brauer monoid, $Br_n$ consists of the set of equivalence classes of such diagrams and is generated by $\{r_i, e_i \mid 1\leq i\leq n-1\}$ (see Figure \ref{figg}). The symmetric group $S_n$, generated by the transpositions $\{r_i\mid 1\leq i \leq n-1\}$, sits inside of $Br_n$. As in the symmetric group, a natural length function $l_{Br}:Br_n\longrightarrow\mathbb{N}$ exists for the Brauer monoid: for $w\in Br_n$, define $l_{Br}(w)$ to be the minimum number of generators ($\{r_i,e_i\}$) needed to express $w$.

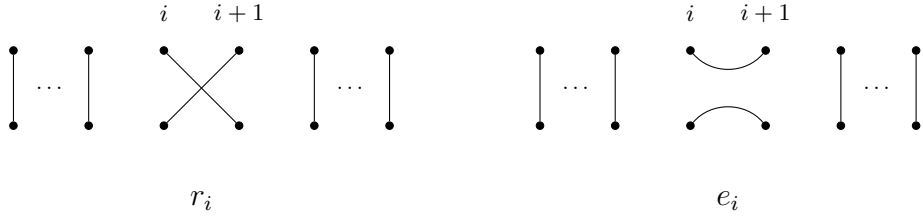
\begin{figure}[H] 
\begin{center}
\begin{tikzpicture}
   \node[pnt] at (-3,0) (v_1) {};
   \node[pnt] (v_2) [right of=v_1] {};
   \node[pnt] (v_3) [right of=v_2] {};
   \node[pnt] (v_4) [right of=v_3] {};
   \node[pnt] (v_5) [right of=v_4] {};
   \node[pnt] (v_6) [right of=v_5] {};
   \node at (-2.5,-.5) {$\dots$};
   \node at (1.5,-.5) {$\dots$};
   \node at (-1,.5) {\small$i$};
   \node at (0,.5) {\small$i+1$};
   \node at (-.5,-2) {\large${r_i}$};
   \node[pnt] (v_7) [below of=v_1] {};
   \node[pnt] (v_8) [below of=v_2] {};
   \node[pnt] (v_9) [below of=v_3] {};
   \node[pnt] (v_10) [below of=v_4] {};
   \node[pnt] (v_11) [below of=v_5] {};
   \node[pnt] (v_12) [below of=v_6] {};
   \node[pnt] at (4,0) (w_1) {};
   \node[pnt] (w_2) [right of=w_1] {};
   \node[pnt] (w_3) [right of=w_2] {};
   \node[pnt] (w_4) [right of=w_3] {};
   \node at (4.5,-.5) {$\dots$};
   \node at (8.5,-.5) {$\dots$};
   \node at (6,.5) {\small$i$};
   \node at (7,.5) {\small$i+1$};
   \node at (6.5,-2) {\large${e_i}$};
   \node[pnt] (w_5) [right of=w_4] {};
   \node[pnt] (w_6) [right of=w_5] {};
   \node[pnt] (w_7) [below of=w_1] {};
   \node[pnt] (w_8) [below of=w_2] {};
   \node[pnt] (w_9) [below of=w_3] {};
   \node[pnt] (w_10) [below of=w_4] {};
   \node[pnt] (w_11) [below of=w_5] {};
   \node[pnt] (w_12) [below of=w_6] {};
   \path[-]
    (v_1) edge node {} (v_7)
    (v_2) edge  node {} (v_8)
    (v_3) edge  node {} (v_10)
    (v_5) edge  node {} (v_11)
    (v_6) edge node {} (v_12)
    (v_4) edge node {} (v_9)
    (w_1) edge node {} (w_7)
    (w_2) edge  node {} (w_8)
    (w_3) edge [bend right=50] node {} (w_4)
    (w_10) edge [bend right=50]  node {} (w_9)
    (w_5) edge  node {} (w_11)
    (w_6) edge node {} (w_12);
\end{tikzpicture}
\caption{$r_i, e_i\in Br_n$}\label{figg}
\end{center}
\end{figure}

The Brauer algebra, $\mathcal{B}r_n$, is the $\mathbb{C}(q)$-algebra with basis $Br_n$.  Equivalently (see, for example \cite{BenkhartShaderRam}), $\mathcal{B}r_n$ has algebraic presentation given by generating set $$\{r_i, e_i \mid 1\leq i\leq n-1\},$$ along with relations:

$$\begin{array}{llll}
(B1)& r_i^2=1, &(B2)& r_ir_j=r_jr_i, \;\;\;\;r_ie_j=e_jr_i,\;\;\;\; e_ie_j=e_je_i,\;\; \\
&&&|i-j|>1\\
 (B3)& e_i^2=qe_i, &(B4)& e_ir_i=r_ie_i=e_i, \\
 (B5)&r_ir_{i+1}r_i=r_{i+1}r_ir_{i+1},& (B6)&e_ie_{i+1}e_i=e_i,\;\;\;\; e_{i+1}e_ie_{i+1}=e_{i+1},\\
(B7)& r_ie_{i+1}e_i=r_{i+1}e_i, & (B8)& e_{i+1}e_ir_{i+1}=e_{i+1}r_i.\\
 \end{array}$$

\subsection{The BMW Algebra}\label{BMWdefs}
Elements of the BMW monoid are realized as generalized Brauer diagrams called \textbf{tangles}. A tangle is again a diagram on $2$ rows of $n$ points each with edges connecting pairs of points regardless of row and each point part of exactly one edge. At each crossing of two edges we distinguish which edge passes above and which passes below (see Figure \ref{BMW}). As in the Brauer monoid, multiplication is concatenation of diagrams and two tangles are  equivalent if they differ only in their number of closed loops.

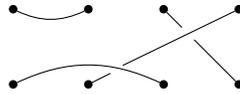
\begin{figure}[H] 
\begin{center}
\begin{tikzpicture}
   \node[pnt] at (-3,0) (v_1) {};
   \node[pnt] (v_2) [right of=v_1] {};
   \node[pnt] (v_3) [right of=v_2] {};
   \node[pnt] (v_4) [right of=v_3] {};
   \node[pnt] (v_5) [below of=v_1] {};
   \node[pnt] (v_6) [below of=v_2] {};
   \node[pnt] (v_7) [below of=v_3] {};
   \node[pnt] (v_8) [below of=v_4] {};
   \node (v) at (-.7,-.3) {};
   \node (w) at (-.65,-.35) {};
   \node (u) at (-1.6,-.8) {};
   \node (vv) at (-1.65,-.8) {};
     \path[-]
    (v_1) edge [bend right=25] node {} (v_2)
    (u) edge  node {} (v_6)
    (v) edge  node {} (v_8)
    (w) edge  node {} (v_3)
    (vv) edge  node {} (v_4)
    (v_7) edge [bend right=25] node {} (v_5);
\end{tikzpicture}
\caption{A Tangle}\label{BMW}
\end{center}
\end{figure}

Further, two tangles are equivalent if they are related by a sequence of Reidemeister moves of type II and III: 

\begin{figure}[H] 
\begin{center}
\begin{tikzpicture}
   \node at (-3,0) (v_1) {};
   \node at (-4,-.5)(r2) {$R_{II}$:};
   \node at (-4,-2.5)(r3) {$R_{III}$:};
  \node (v_2) [right of=v_1] {};
   \node (v_5) at (-3,-1) {};
   \node (q) at (-1,-.5) {$\longleftrightarrow$};
   \node (q) at (-1,-2.5) {$\longleftrightarrow$};
   \node (v_6) [right of=v_5] {};
   \node at (0,0) (w_1) {};
   \node at (.3,-.65) (w) {};
   \node at (.7,-.65) (v) {};
   \node at (.25,-.7) (w1) {};
   \node at (.8,-.7) (w2) {};
   \node (w_2) [right of=w_1] {};
   \node (w_5) at (0,-1) {};
   \node (w_6) [right of=w_5] {};
   \node (u_1) [below of=v_5] {};
   \node (u_3) [right of=u_1] {};
   \node (u_2) at (-2.5,-2) {};
    \node (u_4) [below of=u_1] {};
    \node (u_5) [below of=u_2] {};
    \node (u_6) [below of=u_3] {};
   \node (d_1) [below of=w_5] {};
   \node (d_3) [right of=d_1] {};
   \node (d_2) at (.5,-2) {};
    \node (d_4) [below of=d_1] {};
    \node (d_5) [below of=d_2] {};
    \node (d_6) [below of=d_3] {};
   \node at (-2.65,-2.4) (u1) {};
   \node at (-2.68,-2.35) (u2) {};
   \node at (-2.55,-2.5) (u3) {};
   \node at (-2.65,-2.6) (u4) {};
   \node at (-2.75,-2.75) (u5) {};
   \node at (.55,-2.4) (d1) {};
   \node at (.65,-2.35) (d2) {};
   \node at (.75,-2.2) (d3) {};
   \node at (.4,-2.57) (d4) {};
   \node at (.8,-2.8) (d5) {};
   \node at (.65,-2.65) (d6) {};
     \path[-]
    (v_1) edge [bend right=99] node {} (v_2)
    (v_6) edge [bend right=99] node {} (v_5)
    (w_1) edge [bend right=10]  node {} (w)
    (w_2) edge [bend left=10]  node {} (v)
    (w1) edge  [bend right=10] node {} (w2)
    (w_6) edge [bend right=100] node {} (w_5)
    (u_1) edge node {} (u1)
    (u3) edge node {} (u_6)
    (u3) edge node {} (u2)
    (u4) edge node {} (u_4)
    (u_3) edge node {} (u5)
    (u_2) edge [bend right=40] node {} (u_5)
    (d_1) edge node {} (d5)
    (d6) edge node {} (d_6)
    (d3) edge node {} (d4)
    (d1) edge node {} (d_4)
    (d_3) edge node {} (d2)
    (d_2) edge [bend left=40] node {} (d_5);

\end{tikzpicture}
\caption{Reidemeister Moves II and III}

\end{center}
\end{figure}
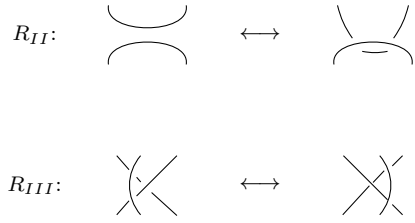

Consider the elements $T_{r_i}$, $T_{r_i}^{-1}$, and $T_{e_i}$  of Figure \ref{gen}. 
\begin{figure}[H] 
\begin{center}
\begin{tikzpicture}
   \node[pnt] at (-3,0) (v_1) {};
   \node[pnt] (v_2) [right of=v_1] {};
   \node[pnt] (v_3) [right of=v_2] {};
   \node[pnt] (v_4) [right of=v_3] {};
   \node[pnt] (v_5) [right of=v_4] {};
   \node[pnt] (v_6) [right of=v_5] {};
   \node at (-2.5,-.5) {$\dots$};
   \node at (1.5,-.5) {$\dots$};
   \node at (-1,.5) {\small$i$};
   \node at (0,.5) {\small$i+1$};
   \node at (-.5,-2) {\large$T_{r_i}$};
   \node[pnt] (v_7) [below of=v_1] {};
   \node[pnt] (v_8) [below of=v_2] {};
   \node[pnt] (v_9) [below of=v_3] {};
   \node[pnt] (v_10) [below of=v_4] {};
   \node[pnt] (v_11) [below of=v_5] {};
   \node[pnt] (v_12) [below of=v_6] {};
   \node[pnt] at (4,0) (w_1) {};
   \node[pnt] (w_2) [right of=w_1] {};
   \node[pnt] (w_3) [right of=w_2] {};
   \node[pnt] (w_4) [right of=w_3] {};
   \node at (4.5,-.5) {$\dots$};
   \node at (8.5,-.5) {$\dots$};
   \node at (6,.5) {\small$i$};
   \node at (7,.5) {\small$i+1$};
   \node at (6.5,-2) {\large$T_{e_i}$};
   \node[pnt] (w_5) [right of=w_4] {};
   \node[pnt] (w_6) [right of=w_5] {};
   \node[pnt] (w_7) [below of=w_1] {};
   \node[pnt] (w_8) [below of=w_2] {};
   \node[pnt] (w_9) [below of=w_3] {};
   \node[pnt] (w_10) [below of=w_4] {};
   \node[pnt] (w_11) [below of=w_5] {};
   \node[pnt] (w_12) [below of=w_6] {};
   \node (v) at (-.5,-.5) {};
   \path[-]
    (v_1) edge node {} (v_7)
    (v_2) edge  node {} (v_8)
    (v_3) edge  node {} (v_10)
    (v_5) edge  node {} (v_11)
    (v_6) edge node {} (v_12)
    (v) edge node {} (v_4)
    (v) edge node {} (v_9)
    (w_1) edge node {} (w_7)
    (w_2) edge  node {} (w_8)
    (w_3) edge [bend right=50] node {} (w_4)
    (w_10) edge [bend right=50]  node {} (w_9)
    (w_5) edge  node {} (w_11)
    (w_6) edge node {} (w_12);
\end{tikzpicture}
\end{center}
\end{figure}

\begin{figure}[H] 
\begin{center}
\begin{tikzpicture}
   \node[pnt] at (-3,0) (v_1) {};
   \node[pnt] (v_2) [right of=v_1] {};
   \node[pnt] (v_3) [right of=v_2] {};
   \node[pnt] (v_4) [right of=v_3] {};
   \node[pnt] (v_5) [right of=v_4] {};
   \node[pnt] (v_6) [right of=v_5] {};
   \node at (-2.5,-.5) {$\dots$};
   \node at (1.5,-.5) {$\dots$};
   \node at (-1,.5) {\small$i$};
   \node at (0,.5) {\small$i+1$};
   \node at (-.5,-2) {\large$T_{r_i}^{-1}$};
   \node[pnt] (v_7) [below of=v_1] {};
   \node[pnt] (v_8) [below of=v_2] {};
   \node[pnt] (v_9) [below of=v_3] {};
   \node[pnt] (v_10) [below of=v_4] {};
   \node[pnt] (v_11) [below of=v_5] {};
   \node[pnt] (v_12) [below of=v_6] {};
     \node (v) at (-.5,-.5) {};
   \path[-]
    (v_1) edge node {} (v_7)
    (v_2) edge  node {} (v_8)
    (v_4) edge  node {} (v_9)
    (v_5) edge  node {} (v_11)
    (v_6) edge node {} (v_12)
    (v) edge node {} (v_3)
    (v) edge node {} (v_10);

\end{tikzpicture}
\caption{$T_{r_i}, T_{e_i}, T_{r_i}^{-1}$}\label{gen}
\end{center}
\end{figure}
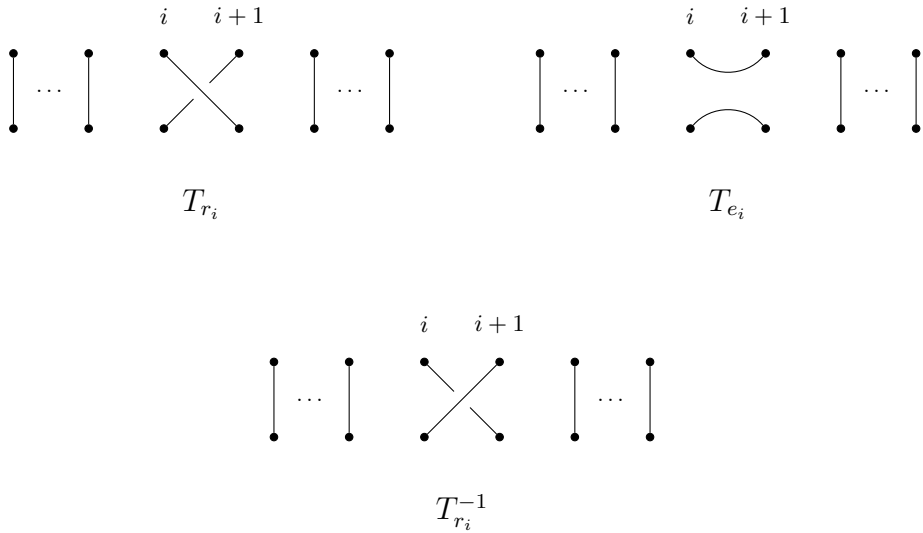
A tangle is \textbf{reachable} if it can be obtained as a finite product of elements from $\{T_{r_i},T_{e_i},T_{r_i}^{-1}\mid 1\leq i \leq n-1\}$. The BMW monoid, $BMW_n$, consists of the set of equivalence classes of reachable tangles on $2n$ points. 

For $m,\ell,q$ parameters satisfying $q=(\ell-\ell^{-1})(m-m^{-1})^{-1}+1$, the $BMW$ algebra, $\mathcal{BMW}_n$, is the $\mathbb{C}(q,m,\ell)$-algebra with basis $BMW_n$ and the following untangling relations: 

\begin{figure}[H] 
\begin{center}
\begin{tikzpicture}
   \node[pnt] (v_3) at (-1,0) {};
   \node[pnt] (v_4) [right of=v_3] {};
   \node[pnt] (v_9) [below of=v_3] {};
   \node[pnt] (v_10) [below of=v_4] {};
   \node (v) at (-.5,-.5) {};
   \node (e) [right of=v] {$=$};
   \node[pnt] (w_3) at (1,0) {};
   \node[pnt] (w_4) [right of=w_3] {};
   \node[pnt] (w_9) [below of=w_3] {};
   \node[pnt] (w_10) [below of=w_4] {};
   \node (w) at (1.5,-.5) {};
   \begin{scope}[shift={(.2,0)}]
   \node (m) at (2.5,-.5) {$+m$};
   \node[pnt] (u_3) at (3,0) {};
   \node[pnt] (u_4) at (3.5,0) {};
   \node[pnt] (u_9) [below of=u_3] {};
   \node[pnt] (u_10) [below of=u_4] {};
   \end{scope}
   \node (u) at (3.5,-.5) {};
   \node (m) [right of=u] {$-m$};
   \node[pnt] (q_3) at (5,0) {};
   \node[pnt] (q_4) [right of=q_3] {};
   \node[pnt] (q_9) [below of=q_3] {};
   \node[pnt] (q_10) [below of=q_4] {};
    \path[-]
    (v_3) edge  node {} (v_10)
    (v) edge node {} (v_4)
    (v) edge node {} (v_9)
    (w_4) edge  node {} (w_9)
    (w) edge node {} (w_3)
    (w) edge node {} (w_10)
    (u_9) edge  node {} (u_3)
    (u_10) edge  node {} (u_4)
    (q_4) edge [bend left=50] node {} (q_3)
    (q_9) edge [bend left=50] node {} (q_10);
\end{tikzpicture}
\end{center}
\end{figure}

\begin{figure}[h]
\begin{center}\includegraphics[scale=.65]{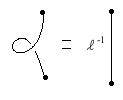}
\end{center}
\begin{center}\includegraphics[scale=.65]{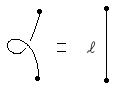}
\end{center}
\begin{center}\includegraphics[scale=.65]{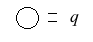}
\end{center}
\caption{Untangling Relations}
\end{figure}

Equivalently (see, for example \cite{GoodmanHauschild}), the BMW algebra has algebraic presentation given by generating set $\{T_{e_i},T_{r_i}, T_{r_i}^{-1}\mid1\leq i \leq n-1\}$, along with relations:
$$\begin{array}{rlrl}
(A1)&T_{e_i}^2=qT_{e_i} , & (A2)&T_{e_i}T_{r_i}=T_{r_i}T_{e_i}=\ell^{-1}T_{e_i}\\
(A3)&T_{e_i}T_{e_{i\pm1}}T_{e_i}=T_{e_i},&(A4) & T_{e_i}T_{r_{i\pm1}}T_{e_i}=\ell T_{e_i},\\
(A5)&T_{r_i}T_{r_{i+1}}T_{r_i}=T_{r_{i+1}}T_{r_i}T_{r_{i+1}},&(A6)&T_{r_i}T_{r_{i\pm1}}T_{e_i}=T_{e_{i\pm 1}}T_{e_{i}}=T_{e_{i\pm1}}T_{r_i}T_{r_{i\pm1}},\\
(A7)&T_{r_i}=T_{r_i}^{-1}+mT_{id}-mT_{e_i}&(A8)&T_{r_i}T_{r_j}=T_{r_j}T_{r_i},\;\;\;\;T_{r_i}T_{e_j}=T_{e_j}T_{r_i},\\
&&&\;\;\;\;T_{e_i}T_{e_j}=T_{e_j}T_{e_i},\;\; \vert i-j\vert>1 , 
\end{array} $$
for $q=(\ell-\ell^{-1})(m-m^{-1})^{-1}+1$ and $T_{id}$ the identity element. For all that follows we let $l=1$.

We map an element of the BMW monoid to the Brauer monoid by `forgetting' crossing information. Denote this map by $\phi:BMW_n\longrightarrow Br_n$.

\begin{example}\label{Brauerimage} For $x$ the tangle of Figure \ref{BMW}, $\phi(x)$ has form:
\begin{figure}[H] 
\begin{center}
\begin{tikzpicture}
   \node[pnt] at (-3,0) (v_1) {};
   \node[pnt] (v_2) [right of=v_1] {};
   \node[pnt] (v_3) [right of=v_2] {};
   \node[pnt] (v_4) [right of=v_3] {};
   \node[pnt] (v_5) [below of=v_1] {};
   \node[pnt] (v_6) [below of=v_2] {};
   \node[pnt] (v_7) [below of=v_3] {};
   \node[pnt] (v_8) [below of=v_4] {};
   \node (v) at (-.7,-.3) {};
   \node (w) at (-.65,-.35) {};
   \node (u) at (-1.6,-.8) {};
   \node (vv) at (-1.65,-.8) {};
     \path[-]
    (v_1) edge [bend right=25] node {} (v_2)
    (v_4) edge  node {} (v_6)
    (v_3) edge  node {} (v_8)
    (v_7) edge [bend right=25] node {} (v_5);
\end{tikzpicture}
\end{center}
\caption{$\phi(x)$}
\end{figure}
\end{example}
Further, each element of the Brauer monoid lifts to the BMW algebra: for $d\in Br_n$, the \textbf{BMW image} of $d$, $T_d$, realizes $d$ as a tangle by redrawing the edges of $d$ from right to left across the first $\lceil\frac{n}{2}\rceil$ points in the bottom row, lifting the pen when crossing an edge that has already been drawn, then moving to the top row of points and drawing all horizontal edges in this row, again lifting the pen when crossing an edge that has already been drawn, and finally drawing the remaining edges of $d$ from right to left across the bottom row of points.
\begin{example} For $d$ the Brauer diagram of Example \ref{Brauerimage}, the BMW image of d is:
\begin{figure}[H] 
\begin{center}
\begin{tikzpicture}
   \node[pnt] at (-3,0) (v_1) {};
   \node[pnt] (v_2) [right of=v_1] {};
   \node[pnt] (v_3) [right of=v_2] {};
   \node[pnt] (v_4) [right of=v_3] {};
   \node[pnt] (v_5) [below of=v_1] {};
   \node[pnt] (v_6) [below of=v_2] {};
   \node[pnt] (v_7) [below of=v_3] {};
   \node[pnt] (v_8) [below of=v_4] {};
   \node (v) at (-.7,-.3) {};
   \node (w) at (-.65,-.35) {};
   \node (u) at (-1.6,-.8) {};
   \node (vv) at (-1.65,-.8) {};
     \path[-]
    (v_1) edge [bend right=25] node {} (v_2)
    (u) edge  node {} (v_6)
    (v_8) edge  node {} (v_3)
    (v_4) edge  node {} (w)
    (v) edge node {} (vv)
    (v_7) edge [bend right=25] node {} (v_5);
\end{tikzpicture}
\end{center}
\caption{$T_d$}
\end{figure}

\end{example}

Note that when $\ell=1$ the BMW image of $d$ has a simple algebraic description.
\begin{definition} For $d\in Br_n$ and $s_i\in\{r_i,e_i\}$, a \textbf{reduced expression} for $d$ is a minimum length expression $d=s_{i_1}s_{i_2}\cdots s_{i_k}$ that has no occurrence of $e_{i+1}r_i$.
\end{definition}
Then the  \textbf{BMW image} of $d$, $T_d$, realizes $d$ as a tangle by setting $$T_d:=T_{s_{i_1}}T_{s_{i_2}}\cdots T_{s_{i_k}},$$
for $d=s_{i_1}s_{i_2}\cdots s_{i_k}$ a reduced expression 

\begin{definition} For $d\in Br_n$ and $e(d)$ the number of $e_i$ terms in a reduced expression for $d$, The \textbf{BMW length} of $T_d$ $L:\mathcal{T}_n\longrightarrow \mathbb{N}$ is given by $$L(T_d)=l'_{Br}(d)+e(d), $$ where $l'_{Br}(d)$ gives the minimum number of generators needed for a reduced expression of $d$.
\end{definition}
\begin{note}\label{refrel} The relations in the Brauer algebra together with the definition of reduced expression ensure that $e(d)$ is well defined. See Table 4 in \cite{brsimplylaced} for the possible rewrites in the Brauer algebra. 
\end{note}
\begin{example}
Let $d=r_3e_2e_1r_3$. Then $T_d=T_{r_3}T_{e_2}T_{e_1}T_{r_3}$ and $L(T_d)=l'_{Br}(d)+2=6$. An alternate reduced expression for $d$ is $d=r_3e_2r_3e_1$, which has the same BMW image by BMW relation (A8): $$T_{r_3}T_{e_2}T_{r_3}T_{e_1}=T_{r_3}T_{e_2}T_{e_1}T_{r_3}.$$

An additional expression for $d$ is $d=r_2e_3r_2e_1$. However, to have a reduced expression we must replace $e_3r_2$:
$$d=r_2e_3r_2e_1=r_2e_3e_2r_3e_1,$$ but then using Brauer relation (B7),
$d=r_3e_2r_3e_1,$ as before.

\end{example}

Theorem 3.12 of \cite{hal} shows that the BMW images of the Brauer monoid elements form a basis for $\mathcal{BMW}_n$. Denote this basis by $\mathcal{T}_n:=\{T_d\mid d\in Br_n\}$.

We consider generation of elements in $\mathcal{T}_n$ via random walks on $\mathcal{T}_n$ and translate these walks into left multiplication in the BMW algebra.

\section{The Random Walk}\label{thewalk}

In the finite group case, left multiplication by a generating set gives rise to a random walk on the group. For example, for each generator $r_i$ of $S_n$, consider the probability distribution
$$   P_i(x,y) = \left\{
     \begin{array}{ll}
       1 &  \text{if}\;y=r_ix,\\
       0 &  \text{else}.
     \end{array}
   \right.
$$
Then for $l_S$ the length function on the symmetric group and $\pi$ given by $$\pi(x)=\frac{\theta^{-l_{S}(x)}}{\displaystyle\sum_{w\in S_n} \theta^{-l_{S}(w)}},$$ 
the Metropolis algorithm construction yields a chain which
interpreted as a random walk on $S_n$ is given by (see \cite{diaram}): 
\begin{equation} \tag{$*$}\label{snwalk}
\begin{split}
& \text{From}\;x\in S_n \;\text{multiply by}\; r_i.\;\text{ If the length increases, move to} \\
& \text{$r_ix$. If the length decreases, flip a}\; \theta\text{-coin and if heads move}\\
& \text{to $r_ix$. If tails, remain at } x.
\end{split}
\end{equation}

\noindent We generalize this walk to the basis of tangles $\mathcal{T}_n$ of the BMW algebra. For $T_d\in\mathcal{T}_n$ and $L$ the length function on $\mathcal{T}_n$ defined in Section \ref{BMWdefs}, let 
$$\pi(T_d)=\frac{\theta^{-L(T_d)}}{\displaystyle\sum_{w\in\mathcal{T}_n}\theta^{-L(w)}},$$

and for $y\in\mathcal{T}_n$ let 
$$P'_i(T_d, y)=\left\{\begin{array}{ll} 1& y=T_{r_id}\\ 0&\text{else}.\end{array}\right.$$

Then the Metropolis algorithm applied to $P'$ with probability distribution $\pi$ yields:
$$   K_i(T_d,y) = \left\{
     \begin{array}{ll}
       1 &  \text{if}\;y=T_{r_id}\;\text{and}\; L(y)\geq L(T_d) ,\\
       \theta &   \text{if}\;y=T_{r_id}\;\text{and}\; L(y)<L(T_d),\\
       1-\theta & \text{if}\; y=T_d.
     \end{array}
   \right.
$$

\begin{remark}\label{Snremark}
 Recall that $S_n\subseteq Br_n$ and note that for $d\in S_n$, $L(T_d)=l'_{Br}(d)=l_{Br}(d)=l_{S}(d)$, where $L,l_{Br}$, and $l_{S}$ denote the length functions on $\mathcal{T}_n,$, $Br_n$, and $S_n$. Then the submatrix of $K_i$ corresponding to states $\{T_d\mid d\in S_n\}$ is exactly the chain $M_i$ of \cite{diaram}.

\end{remark}

\noindent Interpreted as a random walk on $\mathcal{T}_n$, the chain $K_i$ describes the process:
\begin{equation} \tag{$\dagger$}
\begin{split}
& \text{From}\;T_d\in\mathcal{T}_n\text{ consider }d\in Br_n\text{ and multiply by } r_i.\text{ If the}\\ 
&\text{length of the BMW image $T_{r_id}$ increases, move to it. If the}\\
&\text{length decreases, flip a}\; \theta\text{-coin and if heads move to }T_{r_id}. \text{ If}\\ &\text{tails, remain at}\; T_d.
\end{split}
\end{equation}

\noindent In light of Proposition \ref{translate walk} below, this walk can be rephrased as:
\begin{equation} \label{rewalk} \tag{$\dagger\dagger$}
\begin{split}
& \text{From}\;T_d\in\mathcal{T}_n\text{ multiply by }T_{r_i}. \text{ If the result is an element of } \\ 
&\mathcal{T}_n,\text{ move to $T_{r_i}T_d$. Else,} \text{ flip a}\; \theta\text{-coin and if heads move to }\\
& T_{r_i}^{-1}T_d. \text{ If tails, remain at}\; T_d.\\
\end{split}
\end{equation}

\noindent Rephrasing in this way yields the equivalent corresponding Markov chain:

$$   K_i(x,y) = \left\{
     \begin{array}{ll}
       1 &  \text{if}\;y=T_{r_i}x,\\
       \theta &   \text{if}\;y=T_{r_i}^{-1}x,\\
       1-\theta & \text{if}\; y=x.
     \end{array}
   \right.$$
An example of $K_i$ can be found in Appendix \ref{App}. 

\begin{proposition}\label{translate walk} For $T_d\in\mathcal{T}_n$, $$L(T_{r_id})<L(T_d)\iff T_{r_i}T_d\notin\mathcal{T}_n.$$ 
Further, if $T_{r_i}T_d\notin\mathcal{T}_n$, then $T_{r_i}^{-1}T_d=T_{r_id}\in\mathcal{T}_n$, while if $T_{r_i}T_d\in\mathcal{T}_n$, then $T_{r_i}T_d=T_{r_id}$. 
\end{proposition}

\begin{proof}
First write $T_d=T_{s_{i_1}}T_{s_{i_2}}\cdots T_{s_{i_k}}$, for $s_{i_1}\cdots s_{i_k}$ a reduced expression for $d$ with maximum number of $e$ terms. Then $$T_{r_i}T_{d}=T_{r_i}T_{s_{i_1}}T_{s_{i_2}}\cdots T_{s_{i_k}},$$ which, after possibly rearranging using BMW relations (A5) and (A8), has one of the following forms, for some $1\leq j\leq k-2$:

\begin{enumerate}
\item $T_{r_i}T_{d}=T_{s_{i_1}}T_{s_{i_2}}\cdots T_{s_{i_j}}T_{r_i}T_{s_i}T_{s_{i_{j+2}}}T_{s_{i_{j+3}}}\cdots T_{s_{i_k}}$ 
\item $T_{r_i}T_{d}=T_{s_{i_1}}T_{s_{i_2}}\cdots T_{s_{i_j}}T_{r_i}T_{s_{i\pm 1}}T_{s_i}T_{s_{i_{j+3}}}\cdots T_{s_{i_k}},$
\item $T_{r_i}T_{d}=T_{r_i}T_{s_{i_1}}T_{s_{i_2}}\cdots T_{s_{i_k}}, \vert i_1-i\vert >1.$ 
\end{enumerate}

The proof reduces to checking each possible case. For example, if in case (1) with $s_is_{i_{j+2}}=r_ie_{i\pm 1}$,
$$T_{r_i}T_d=T_{s_{i_1}}T_{s_{i_2}}\cdots T_{s_{i_j}}T_{r_i}T_{r_i}T_{e_{i\pm 1}}T_{s_{i_{j+3}}}\cdots T_{s_{i_k}}.$$
Since $T_{r_i}T_{r_i}\notin\mathcal{T}_n$, we see that $T_{r_i}T_d\notin \mathcal{T}_n$. Further, since BMW relations (A5) and (A8) hold in the Brauer monoid,
$$r_id=r_is_{i_1}\cdots s_{i_k}=s_{i_1}s_{i_2}\cdots s_{i_j}r_ir_ie_{i\pm 1}s_{i_{j+3}}\cdots s_{i_k},$$
which by Brauer relation (B1) gives
$$r_id=s_{i_1}\cdots s_{i_j}e_{i\pm 1}s_{i_{j+3}}\cdots s_{i_k},$$ a reduced expression for $r_id$. Thus $l'_{Br}(r_id)=k-1.$ By Note \ref{refrel} all reduced expressions  have the same number of $e$ terms, so $e(r_id)=e(d)$. Hence, $$L(T_{r_id})<L(d).$$


For the second statement, note that $$T_{r_i}^{-1}T_d=T_{s_{i_1}}T_{s_{i_2}}\cdots T_{s_{i_j}}T_{r_i}^{-1}T_{r_i}T_{e_{i\pm 1}}T_{s_{i_{j+3}}}\cdots T_{s_{i_k}}=T_{r_id}.$$

The remaining cases are checked similarly.
\end{proof}

In \cite{diaram}, Diaconis and Ram translate the Markov chain arising from (\ref{snwalk}) into left multiplication by Hecke algebra elements on a suitably chosen basis. Similarly, we translate the chains $K_i$ arising from the Metropolis construction into left multiplication by BMW algebra elements on the basis $\mathcal{T}_n$. 

Define $\mathscr{T}_{r_i}, \mathscr{T}_{e_i}:\mathcal{T}_n\longrightarrow\mathcal{BMW}_n$ as follows: for $x\in\mathcal{T}_n$, 

$$\begin{array}{l}
\mathscr{T}_{r_i}(x)=T_{r_i}x\\
\\
\\\mathscr{T}_{e_i}(x)=\left\{
     \begin{array}{ll}
       T_{e_i}x &  \text{if}\;T_{r_i}x\notin\mathcal{T}_n,\\
       T_{r_i}x &  \text{else}.
     \end{array}
   \right.\\
\end{array}$$

\begin{theorem}\label{leftmult}[Theorem \ref{leftmult1}] Let $Br_n$ be the Brauer monoid and $\mathcal{BMW}_n(m,l)$ the BMW algebra with basis $\mathcal{T}_n=\{T_d\mid d\in Br_n\}$. Let $m=(1-\theta)(\theta)^{-1}$ and $\ell=1$. Then the chain $K_i$ is the same as the matrix of left multiplication by 
$$\theta\mathscr{T}_{r_i}+(1-\theta)\mathscr{T}_{e_i},$$
with respect to the basis $\mathcal{T}_n$ of $\mathcal{BMW}_n$.
\end{theorem} 
\begin{proof}

Let $x\in\mathcal{T}_n$ and consider left multiplication by $T_{r_i}$. If $T_{r_i}x\in\mathcal{T}_n$,
$$(\theta\mathscr{T}_{r_i}+(1-\theta)\mathscr{T}_{e_i})x= \theta T_{r_i}x+(1-\theta)T_{r_i}x=T_{r_i}x.$$

If $T_{r_i}x\notin\mathcal{T}_n$ then by BMW Relation (A7), 
$$T_{r_i}x=(T_{r_i}^{-1}+m T_{id}- ml^{-1}T_{e_i})x=T_{r_i}^{-1}x+(1-\theta)(\theta)^{-1} x- (1-\theta)(\theta)^{-1} T_{e_i}x.$$

By Proposition \ref{translate walk}, $T_{r_i}^{-1}x\in\mathcal{T}_n$, and 
$$\begin{array}{ll}
(\theta\mathscr{T}_{r_i}+(1-\theta)\mathscr{T}_{e_i})x&=  \theta T_{r_i}^{-1}x+(1-\theta) x. 
\end{array}$$
\end{proof}

The chains $K_i$ provide scanning strategies for generating elements of the BMW and Brauer monoids:
$$\begin{array}{ll}
      \displaystyle \frac{1}{n-1}\sum_{i=1}^{n-1} K_i & \text{(random scan)},\\
       K_1K_2\cdots K_{n-1}K_{n-1}\cdots K_2K_1 & \text{(short systematic scan)},\\
       (K_1\cdots K_{n-1}K_{n-1}\cdots K_1)\cdots (K_1K_2K_2K_1)(K_1K_1) & \text{(long systematic scan)}.\\
     \end{array}
$$ 
Theorem \ref{leftmult}, coupled with the results of Section \ref{walkanalysis}, allows for the study of the rate of convergence of the systematic scans arising from the chains $K_i$ using Fourier analysis on the BMW algebra.

\section{Analysis of the Walk}\label{walkanalysis}
Let $K$  denote the matrix corresponding to any of the three scans (random, short systematic, long systematic), as the results of this section hold true for all three scans.

Note that $K$ is Markov and recall that a communication class $C$ of a Markov chain is closed if for each state $x\in C$ and for all $y\notin C$, $y$ is not accessible from $x$. We determine the closed communication classes of $K$ and analyze the stationary distribution of each closed communication class.

The communication classes of $K$ depend on the number of \emph{lower horizontal edges} in the tangle diagrams for the states.

\begin{definition} Let $x\in\mathcal{T}_n$. An edge of $x$ is \textbf{lower (respectively, upper) horizontal} if it connects two points that are both on the bottom (respectively, top) row of the diagram of $x$.\end{definition}
\begin{example} In Figure \ref{e's}, $E_3$ is the only lower horizontal edge and $E_1$ is the only upper horizontal edge.
\begin{figure}[H] 
\begin{center}
\begin{tikzpicture}
   \node[pnt] at (-3,0) (v_1) {};
   \node[pnt] (v_2) [right of=v_1] {};
   \node[pnt] (v_3) [right of=v_2] {};
   \node[pnt] (v_4) [right of=v_3] {};
   \node[pnt] (v_5) [below of=v_1] {};
   \node[pnt] (v_6) [below of=v_2] {};
   \node[pnt] (v_7) [below of=v_3] {};
   \node[pnt] (v_8) [below of=v_4] {};
   \node (v) at (-.7,-.3) {};
   \node (w) at (-.65,-.35) {};
   \node (u) at (-1.6,-.8) {};
   \node (vv) at (-1.65,-.8) {};
     \path[-]
    (v_1) edge [bend right=25] node {$E_1$} (v_2)
    (u) edge  node {} (v_6)
    (v) edge  node {$E_2$} (v_8)
    (w) edge  node {} (v_3)
    (vv) edge  node {$E_4$} (v_4)
    (v_7) edge [bend right=25] node {$E_3$} (v_5);
\end{tikzpicture}
\caption{}\label{e's}
\end{center}
\end{figure}

\end{example}

Note that left multiplication by $T_{r_i}, T_{r_i}^{-1}$ does not affect existing lower horizontal edges in a tangle diagram, nor can it create new ones.  As $K$ is determined by left multiplication by $T_{r_i}, T_{r_i}^{-1}$,  the communication classes of $K$ consist of states with common lower horizontal edges. For $x_i\in\mathcal{T}_n$, let $\mathbf{X}_i$ denote its communication class:
$$\mathbf{X}_i:=\{y\in\mathcal{T}_n\mid\text{lower horizontal edges of $y$ the same as those of $x_i$}\}.$$

For each communication class $\mathbf{X}_i$, let $[K]_i$ denote the corresponding submatrix of $K$. Note that the communication class for $x_0:=T_{id}$ consists of the states $\{T_d\mid d\in S_n\}$. Then by  Remark \ref{Snremark}, $[K_i]_0=M_i$, and so $[K]_0$ can be analyzed using the methods of \cite{diaram}. For the remainder of the paper we consider the remaining communication classes of $K$.

To analyze the time to stationarity of the submatrix $[K]_1$ corresponding to a communication class $\mathbf{X}_1$, we pair $\mathbf{X}_1$ with a communication class, $\mathbf{X}_2$, whose states have the same number of lower horizontal edges as those in $\mathbf{X}_1$. For $w\in\mathbf{X}_1$, let $w^*$ denote the element of $\mathbf{X}_2$ with the same upper configuration as $w$. Define the matrix:
$$\tilde{K}(x,y) = \left\{
     \begin{array}{cl}
       K(x,y) &  \text{if}\;x,y\in\mathbf{X}_1\;,\\
       K(x,y) & \text{if } x=w^*,y=z^* \text{ for } w,z\in\mathbf{X}_1,\\ 
        1 &  \text{if}\;x=y, x\notin\mathbf{X}_1\cup\mathbf{X}_2,\\
        0 &  \text{else}.\\
       \end{array}
   \right.
$$

\begin{example}\label{exmat} For $\mathcal{T}_3\subseteq \mathcal{BMW}_3$, let $x_1=T_{e_1}$ and $x_2=T_{e_1}T_{r_2}$, so $\mathbf{X}_1=\{T_{e_1}, T_{r_2}T_{e_1}, T_{e_2}T_{e_1}\}$ and $\mathbf{X}_2=\{T_{e_1}T_{r_2}, T_{r_2}T_{e_1}T_{r_2}, T_{e_2}T_{e_1}T_{r_2}\}.$ Note that $T_{e_1}^*=T_{e_1}T_{r_2}$, while $T_{r_2}T_{e_1}^*=T_{r_2}T_{e_1}T_{r_2}$ and $T_{e_2}T_{e_1}^*=T_{e_2}T_{e_1}T_{r_2}$.

Then for $K=\frac{1}{2}(K_1+K_2)$, 
$$2[K]_1=\bordermatrix{~ & T_{e_1} & T_{r_2}T_{e_1}&T_{e_2}T_{e_1} \cr
              ~ & 1 & \theta&0 \cr
              ~ & 1 & 1-\theta &\theta \cr
              ~& 0 & 1& 2-\theta \cr}, $$                  
                  
                  $$2[K]_2=\bordermatrix{~ &T_{e_1}T_{r_2}&T_{r_2}T_{e_1}T_{r_2}&T_{e_2}T_{e_1}T_{r_2}\ \cr
              ~ & 1 & \theta&0 \cr
              ~ & 1 & 1-\theta &\theta \cr
              ~& 0 & 1& 2-\theta \cr}. $$

\vspace{.5cm}
Then $\tilde{K}=[K]_1\bigoplus [K]_2\bigoplus I_9$, for $I_9$ the $9\times 9$ identity matrix.
\end{example}

Let $\pi$ denote the stationary distribution of $\tilde{K}$ and for $T_x\in\mathcal{T}_n$ let $[\pi]_x$ denote the column of $\pi$ corresponding to $T_x$:
$$[\pi]_x:=\sum_{T_y\in\mathcal{T}_n}\pi_x(y)T_y.$$
Note that $\pi_x(y)$ represents the probability of ending at state $T_y$ after starting at $T_x$. To analyze the time to stationarity of $\tilde{K}$ we consider the total variation norm: 
\begin{equation}\label{number} |\tilde{K}^m_x-\pi|_{TV}.\end{equation}

We bound the total variation norm using a trace norm on $\mathcal{BMW}_n$.
\begin{definition} Define $\tilde{\tau}:\mathcal{T}_n\rightarrow\mathbb{C}$ as follows: for $x\in\mathcal{T}_n$,
$$\tilde{\tau}(x) = \left\{
     \begin{array}{cl}
        1 &  \text{if}\;x=T_{id} ,\\
        0 &  \text{else},\\
       \end{array}
   \right.
$$ The \textbf{restricted trace}, $\tau: \mathcal{BMW}_n\rightarrow  \mathbb{C}$, is the linear extension of $\tilde{\tau}$ to $\mathcal{BMW}_n$.
\end{definition}
\begin{proposition}\label{oftenzero} For $T_x,T_y\in\mathcal{T}_n$, $\tau(T_xT_y) = \left\{
     \begin{array}{cl}
        1 &  \text{if}\;x=y^{-1} ,\\
        0 &  \text{else}.\\
       \end{array}
   \right.
   $
\end{proposition}
\begin{corollary}\label{istrace} $\tau$ is a trace function on $\mathcal{BMW}_n$.
\end{corollary}
\begin{proof}[Proof of Proposition \ref{oftenzero}] Let $T_x, T_y\in \mathcal{T}_n$. Then $T_x=T_{s_{j_1}}\cdots T_{s_{j_k}}$, where for each $1\leq j\leq k$, $T_{s_{i_j}}\in\{T_{r_i},T_{e_i}\mid 1\leq i\leq n-1\}$. First note by the BMW relations (A1)-(A8) that if for some $1\leq i\leq n-1$, $T_{e_i}$ is a factor of $T_x$, then each term of the product $T_xT_y$ has at least one $T_{e_i}$ factor. Hence, no term in the product $T_xT_y$ is the identity, so $\tau(T_xT_y)=0$. Similarly, $\tau(T_yT_x)=0$. 

Thus, if $T_{s_{j_l}}=T_{e_i}$ for some $1\leq l\leq k$, $1\leq i\leq n-1$, then $\tau(T_xT_y)=\tau(T_yT_x)=0$ for all $T_y\in\mathcal{T}_n$. Equivalently, $\tau(T_xT_y)=0$ for all $x\in Br_n- S_n$, $y\in Br_n$. 

Next note that $T_x\in \mathcal{T}_n$ has an inverse iff $x\in S_n\subset Br_n$. Hence we need show for $x,y\in S_n$ that $$\tau(T_xT_y) = \left\{
     \begin{array}{cl}
        1 &  \text{if}\; x=y^{-1} ,\\
        0 &  \text{else}.\\
       \end{array}
   \right.
   $$

But note that $\tau|_{S_n}$ is just a scalar multiple of the trace function $\vec{t}$ on the Iwahori Hecke algebra of $S_n$ (See e.g. \cite{diaram}[Section 3]). 

\end{proof}

Thus $\tau$ is a trace function on $\mathcal{BMW}_n$ with $\tau(T_xT_y)=0$ for all $x,y\in Br_n- S_n$. In fact, $\tau$ extends the natural trace function of the Hecke algebra, $\mathcal{H}_n$, viewing $\mathcal{H}_n$ as a subalgebra of $\mathcal{BMW}_n$. We analyze $\tilde{K}$ using the bilinear form
arising from $\tau$, which reformulates questions about the time to stationarity in terms of the representation theory of the underlying Hecke subalgebra of $\mathcal{BMW}_n$. 

Recall that $\tilde{K}$ consists of two submatrices corresponding to two communication classes $\mathbf{X}_1$ and $\mathbf{X}_2$ of $K$. Note that for each $T_x\in\in \mathbf{X}_1\cup\mathbf{X}_2$, $x\in Br_n-S_n$. Thus, $\tau(T_xT_y)=0$ for all $T_y\in \mathcal{T}_n$. In order for $\tau$ to be nontrivial on the communication classes of $\tilde{K}$, we rewrite $\tilde{K}$ with respect to a shifted basis for $\mathcal{BMW}_n$. 

\begin{definition}\label{basisshift} Let $\pi$ denote the stationary distribution of $\tilde{K}$. To each $T_x\in\mathbf{X}_1$, associate a distinct $s_x\in S_n$ such that $s_x\neq s_y^{-1}$ for all $T_y\in\mathbf{X_1}$ and $s_x$ has order greater than 2.  For $T_x\in \mathbf{X}_1$ and for $T_y\notin\mathbf{X}_1\cup\mathbf{X}_2$, let
\begin{equation}
\begin{split}
&\hat{T}_x:=T_x+\pi_x(x)^{-\frac{1}{2}}T_{s_x},\\
&\hat{T}_{x^*}:=T_{x^*}+\pi_x(x)^{-\frac{1}{2}}T_{{s_x}^{-1}}=T_{x^*}+\pi_{x^*}(x^*)^{-\frac{1}{2}}T_{{s_x}^{-1}},\\
&\hat{T}_y:=T_y.
\end{split}
\end{equation}

\end{definition}

\begin{note}
 By construction, $\pi_x(x)=\pi_{x^*}(x^*)$ for all $x\in \mathbf{X}_1$.
\end{note}
\begin{note} In Appendix \ref{appcom} we show that $S_n$ contains enough distinct elements to make the associations of Definition \ref{basisshift} for all communication classes corresponding to elements with at least two lower horizontal edges.  The remaining communication classes are analyzed separately through techniques discussed in Appendix \ref{appcom}.
\end{note} 
For the remainder of this section let $\mathbf{X}_1$ be a communication class whose elements contain at least two lower horizontal edges. 

\begin{lemma} $\hat{\mathcal{T}}_n:=\{\hat{T}_x\mid x\in Br_n\}$ is a basis for $\mathcal{BMW}_n$.
\end{lemma}
Now let $\langle \;,\;\rangle_{\mathcal{BMW}}$ denote the trace form of Section \ref{planchinv}


\begin{lemma}\label{innerprodeval} For $T_x\in\mathbf{X}_1\cup\mathbf{X}_2$ and $y\in Br_n$, 
$$\langle \hat{T}_x,\hat{T}_y\rangle_{\mathcal{BMW}} = \left\{
     \begin{array}{cl}
        \pi_x(x)^{-1} &  \text{if}\; y=x^* ,\\
        \pi_x(x)^{-\frac{1}{2}} &  \text{if}\; y=(s_x)^{-1},\\
       0 &  \text{else},\\
       \end{array}
   \right. $$
while
$$\langle \hat{T}_{s_x},\hat{T}_y\rangle_{\mathcal{BMW}} = \left\{
     \begin{array}{cl}
        \pi_x(x)^{-\frac{1}{2}} &  \text{if}\; y=x^* ,\\
        1 &  \text{if}\; y=(s_x)^{-1},\\
       0 &  \text{else}.\
       \end{array}
   \right. $$\end{lemma}
\begin{proof} Follows from Proposition \ref{oftenzero} and the linearity of trace.
\end{proof}
Let $\hat{K}$ be the matrix of $\tilde{K}$ with respect to $\hat{\mathcal{T}}_n$. Note that time to stationarity is invariant under change of basis.

\begin{lemma} \label{Khatstruct} For $\hat{T}_x\in\hat{\mathcal{T}}_n$,
\begin{enumerate}
\item If $\hat{T}_x\in\hat{\mathbf{X}}_1$, 
$$ \hat{K}(\hat{T}_x,\hat{T}_y) = \left\{
     \begin{array}{ll}
     K(T_x,T_y) &  \text{if}\;\hat{T}_y\in\hat{\mathbf{X}}_1\cup\hat{\mathbf{X}}_2,\\
     (1-K(T_x,T_x))\pi_x(x)^{-\frac{1}{2}}&\text{if}\; y=s_x,\\
     -K(T_x,T_y)\pi_y(y)^{-\frac{1}{2}} &  \text{if}\;y=s_z, z\neq x, z\in\mathbf{X}_1\\
       0& \text{else}, \\ 
      \end{array}
   \right.
$$ and similarly for $\hat{T}_x\in\hat{\mathbf{X}}_2$.
\item If $\hat{T}_x\notin\hat{\mathbf{X}}_1\cup\hat{\mathbf{X}}_2$, 
$$ \hat{K}(\hat{T}_x,\hat{T}_y) = \left\{
     \begin{array}{ll}
     1 &  \text{if}\;y=x,\\
      0& \text{else}. \\ 
      \end{array}
   \right.
$$

\end{enumerate}
\end{lemma}
\begin{proof} Follows from definition of $\hat{K}$ and $\hat{\mathcal{T}}_n$.
\end{proof}

Lemma \ref{Khatstruct} shows that $\hat{K}$ is a direct sum $\hat{K}_1\bigoplus\hat{K}_2\bigoplus I_{\hat{m}}$, where for $i=1,2$, the matrix $\hat{K}_i$ corresponds to $\{\hat{T}_x,\hat{T}_{s_x}\mid T_x\in\mathbf{X}_i\}$,  and $\hat{m}=|\mathcal{T}_n|-4|\mathbf{X}_1|$. Further,  
$$\hat{K}(\hat{T}_x,\hat{T}_y)=\tilde{K}(T_x,T_y)=K(T_x,T_y),$$
for all $T_x,T_y\in\mathbf{X}_1\cup\mathbf{X}_2$.

Recall that $\pi$ denotes the stationary distribution of $\tilde{K}$. For $T_x\notin\mathbf{X}_1\cup\mathbf{X}_2$, $\pi_x(y)=0$ for all $T_y\neq T_x$, and so  \begin{equation}\label{eqn1}[\pi]_x=T_x.\end{equation} 
Further, for $T_x\in\mathbf{X}_1$, $\pi_x(y)=0$ for all $T_y\notin\mathbf{X}_1$, and so
\begin{equation}\label{eqn2}[\pi]_{x}=\sum_{T_y\in\mathbf{X}_1} \pi_{x}(y)T_{y},\end{equation}
and similarly for $\mathbf{X}_2$.

Let $\hat{\pi}$ denote the stationary distribution of $\hat{K}$ and $[\hat{\pi}]_x$ the stationary distribution of $\hat{K}$ corresponding to column $\hat{T}_x$. Let $\hat{\mathbf{X}}_i=\{\hat{T}_x\mid T_x\in \mathbf{X}_i\}$. 

\begin{lemma}\label{khatstat} Let $\pi$ be the stationary distribution of $\tilde{K}$ and $\hat{\pi}$ the stationary distribution of $\hat{K}$.
\begin{enumerate}
\item  For $\hat{T}_{x}\in\hat{\mathbf{X}}_1$,
$$ [\hat{\pi}]_{x} = \sum_{\hat{T}_y\in\mathbf{X}_1}(\pi_x(y)\hat{T}_y-\pi_x(y)^{\frac{1}{2}}\hat{T}_{s_y})+\pi_x(x)^{-\frac{1}{2}}\hat{T}_{s_x},$$

and similarly for $\hat{T}_{x}\in\hat{\mathbf{X}}_2$.
\item If $\hat{T}_y\notin \hat{\mathbf{X}}_1\cup\hat{\mathbf{X}}_2$,  $[\hat{\pi}]_y= \hat{T}_y$
\end{enumerate}
\end{lemma}
\begin{proof} Part (2) follows from Lemma \ref{Khatstruct}. To prove (1), note that for $\hat{T}_x\in\hat{\mathbf{X}}_1$, $[\hat{\pi}]_{x}=[\pi]_{x}+\pi_{x}(x)^{-\frac{1}{2}}[\pi]_{s_{x}}$. Then by equations (\ref{eqn1}) and (\ref{eqn2}),
\begin{equation*}
\begin{split}
[\hat{\pi}]_{x}&=[\pi]_{x}+\pi_{x}(x)^{-\frac{1}{2}}[\pi]_{s_{x}}\\
&=\sum_{T_{y}\in\mathbf{X}_1} \pi_{x}(y)T_{y}+\pi_{x}(x)^{-\frac{1}{2}}T_{s_{x}}\\
&=\sum_{T_{y}\in\mathbf{X}_1} \left( \pi_{x}(y)(T_{y}+\pi_{y}(y)^{-\frac{1}{2}}T_{s_{y}})-\pi_{x}(y)\pi_y(y)^{-\frac{1}{2}}T_{s_{y}}\right)+\pi_{x}(x)^{-\frac{1}{2}}T_{s_{x}}.\\
&=\sum_{\hat{T}_{y}\in\hat{\mathbf{X}}_1 } \left( \pi_{x}(y)\hat{T}_{y}-\pi_{x}(y)^{\frac{1}{2}}\hat{T}_{s_{y}}\right)+\pi_{x}(x)^{-\frac{1}{2}}\hat{T}_{s_{x}}.
\end{split}
\end{equation*}
\end{proof}

For $T_x\in\mathbf{X}_1\cup\mathbf{X}_2$, Lemma \ref{khatstat} shows that $\pi_x(y)=\hat{\pi}_x(y)$ for all $T_{y}\in\mathbf{X}_1\cup\mathbf{X}_2$. However, $\pi_x(s_{y})=0$, but $\hat{\pi}_{x}(s_{x})=\pi_{x}(x)^{-\frac{1}{2}}-\pi_{x}(x)^{\frac{1}{2}}$ and for $y\neq x$, $\hat{\pi}_{x}(s_{y})=-\pi_{x}(y)^{\frac{1}{2}}$. 

Let $\hat{\mathcal{S}}:=\{\hat{T}_{s_x}\mid s_x\in\mathcal{S}\}$. Consider  the $L^2(\hat{\pi})$-norm restricted to the subspace generated by $\hat{\mathbf{X}}_1\cup \hat{\mathbf{X}}_2\cup\hat{\mathcal{S}}$:
\begin{definition} For functions $f,g:\hat{\mathbf{X}}_1\cup\hat{\mathbf{X}}_2\cup \hat{\mathcal{S}}\rightarrow\mathbb{C}$, let 
$$\langle f,g\rangle_2:=\sum_{\hat{T}_x\in\hat{\mathbf{X}}_1\cup\hat{\mathbf{X}}_2\cup\hat{\mathcal{S}}} f(x)g(x)\hat{\pi}_x(x).$$
 \end{definition}

For $m\in\mathbb{N}$, let $[\hat{K}^m]_x$ denote the column of $\hat{K}^m$ corresponding to $\hat{T}_x$:
$$[\hat{K}^n]_x=\sum_{{\hat{T}_{x_i}}\in\hat{\mathbf{X}}_1\cup\hat{\mathbf{X}}_2\cup\hat{\mathcal{S}}}K_x^n(x_i)\hat{T}_{x_i}.$$ 
To find the time to stationarity of $\hat{K}$ (and hence $\tilde{K}$ and $K$), we analyze $\| [\hat{K}^m]_x-[\hat{\pi}]_x\|_2$.

\begin{lemma}\label{translateinnerprod} Let $f,g$ be complex-valued functions on $\hat{\mathbf{X}}_1\cup\hat{\mathbf{X}}_2\cup\hat{\mathcal{S}}$ and let $*:\hat{\mathbf{X}}_1\cup\hat{\mathbf{X}}_2\cup\hat{\mathcal{S}}\rightarrow\hat{\mathbf{X}}_1\cup\hat{\mathbf{X}}_2\cup\hat{\mathcal{S}}$ be the involution that sends $\hat{T}_x$ to $\hat{T}_{x^*}$ for $\hat{T}_x\in\hat{\mathbf{X}}_1\cup\hat{\mathbf{X}}_2$, and $\hat{T}_{s_x}$ to $\hat{T}_{s_x^{-1}}$ for $\hat{T}_{s_x}\in\hat{\mathcal{S}}$. Then for $\langle\;,\;\rangle_{BMW}$ the bilinear form arising from the trace $\tau$,
$$\langle f/\hat{\pi},g/\hat{\pi}\rangle_2=\langle f,g^*\rangle_{BMW}-\sum_{\hat{T}_{s_x}\in\hat{\mathcal{S}}}\frac{f(x)g(s_x^{-1})+f(s_x^{-1})g(x^*)}{\hat{\pi}_x(x)^{\frac{1}{2}}}.$$
\end{lemma}

\begin{proof} By Lemma \ref{innerprodeval}, 
\begin{equation*}
\begin{split}
\langle f/\hat{\pi},g/\hat{\pi}\rangle_2&= \sum \frac{f(x)g(x)}{\hat{\pi}_x(x)}\\
&=\sum_{\hat{T}_x\in \hat{\mathbf{X}}_1\cup\hat{\mathbf{X}}_2\cup\hat{\mathcal{S}}} f(x)g(x)\langle \hat{T}_x, (\hat{T}_{x})^*\rangle_{BMW}\\
&=\sum_{\hat{T}_x,\hat{T}_y\in \hat{\mathbf{X}}_1\cup\hat{\mathbf{X}}_2\cup\hat{\mathcal{S}}} f(x)g(y)\langle \hat{T}_x, (\hat{T}_{y})^*\rangle_{BMW}-\sum_{\hat{T}_{s_x}\in\hat{\mathcal{S}}}\frac{f(x)g(s_x^{-1})}{\hat{\pi}_x(x)^{\frac{1}{2}}}\\
&\;\;\;\;\;\;\;\;-\sum_{\hat{T}_{s_x}\in\hat{\mathcal{S}}} \frac{f(s_x^{-1})g(x^*)}{\hat{\pi}_x(x)^{\frac{1}{2}}}\\
&=\langle f,g^*\rangle_{BMW}-\sum_{\hat{T}_{s_x}\in\hat{\mathcal{S}}}\frac{f(x)g(s_x^{-1})+f(s_x^{-1})g(x^*)}{\hat{\pi}_x(x)^{\frac{1}{2}}}.
\end{split}
\end{equation*}
\end{proof}

\begin{corollary} For $\hat{T}_x\in\hat{\mathbf{X}}_1$, 
$$\langle [\hat{K}^m/\hat{\pi}]_{x},[\hat{K}^m/\hat{\pi}]_{x}\rangle_2=\langle [\hat{K}^m]_{x},[\hat{K}^m]_{x}\rangle_{BMW}-\sum_{\hat{T}_y\in\mathbf{X}_2}\frac{\hat{K}^m_{x}(s_y^{-1})\hat{K}^m_{x}(y^*)}{\hat{\pi}_{y}(y)^{\frac{1}{2}}}.$$
\end{corollary}

$K$ is Markov, so there exists $N\in\mathbb{N}$ with $\hat{K}^m_{x}\geq 0$ for all $m>N$. Further, $\hat{\pi}$ is the stationary distribution of a Markov chain, so $\hat{\pi}_{y}(y)\geq 0$. We can thus bound the time to stationarity by the BMW trace.
\begin{theorem}[Theorem \ref{main2}] For $\hat{T}_x\in \hat{\mathbf{X}}_1\cup\hat{\mathbf{X}}_2$, 
$$\langle [\hat{K}^n/\hat{\pi}]_{x},[\hat{K}^n/\hat{\pi}]_{x}\rangle_2\leq \langle [\hat{K}^n]_{x},[\hat{K}^n]_{x}\rangle_{BMW}.$$ Hence, $$\| [\hat{K}^n/\hat{\pi}]_{x}-1\|^2_2 \leq \|[\hat{K}^n]_{x}-1\|_{BMW}^2.$$
\end{theorem}

Thus, studying the time to stationarity of $\hat{K}$ can be achieved by studying $$\|[\hat{K}^n]_{x}-1\|_{BMW}^2.$$

\section*{Acknowledgments}
The author would like to especially thank Arun Ram and Dan Rockmore for many helpful and encouraging conversations.

\appendix

\section{Example of Walk in $\mathcal{BMW}_3$}\label{App}
\begin{example} In $\mathcal{BMW}_3$, $$\mathcal{B}_3=\mathbf{R}\cup\mathbf{E}_1\cup\mathbf{E}_2\cup\mathbf{E}_3,$$ for $$\begin{array}{ll}
\mathbf{R}=\{T_{id},T_{r_1},T_{r_2},T_{r_1}T_{r_2}, T_{r_2}T_{r_1},T_{r_1}T_{r_2}T_{r_1}\},&\mathbf{E}_1=\{T_{e_1},T_{r_2}T_{e_1},T_{e_2}T_{e_1}\},\\
\mathbf{E}_2=\{T_{e_2},T_{r_1}T_{e_2}, T_{e_1}T_{e_2}\},&
\mathbf{E}_3=\{T_{e_1}T_{r_2},T_{r_2}T_{e_1}T_{r_2},T_{e_2}T_{e_1}T_{r_2}\} .
\end{array}$$
The Markov chain $K_1$ has form $$R\bigoplus E_1\bigoplus E_2\bigoplus E_3,$$ for 

$$R=\bordermatrix{~ &{id}&{r_1}&{r_2}&{r_1}{r_2}& {r_2}{r_1}&{r_1}{r_2}{r_1}\ \cr
id & 0& \theta & 0  & 0& 0   & 0  &  \cr
r_1&1&1-\theta&0&0&0&0 \cr
r_2&0&0&0&\theta&0&0 \cr
r_1r_2~&0&0&1&1-\theta&0&0 \cr
r_2r_1~&0&0&0&0&0&\theta \cr
r_1r_2r_1~&0&0&0&0&1&1-\theta \cr},\; $$ $$E_1=\bordermatrix{~&{e_1}&{r_2}{e_1}&{e_2}{e_1}\ \cr
e_1&1&0&0 \cr
r_2e_1&0&0&\theta \cr
e_2e_1&0&1&1-\theta \cr},
\;E_2=\bordermatrix{~&{e_2}&r_1{e_2}&e_1{e_2}\ \cr
e_2&0&\theta&0 \cr
r_1e_2&1&1-\theta&0 \cr
e_1e_2&0&0&1 \cr},$$
$$E_3=\bordermatrix{~& {e_1}{r_2}&{r_2}{e_1}{r_2}&{e_2}{e_1}{r_2}\ \cr
e_1r_2&1& 0& 0 \cr
r_2e_1r_2&0& 0& \theta \cr
e_2e_1r_2&0& 1& 1-\theta \cr}.$$

\end{example}

\section{Symmetric Group Elements}\label{appcom}

\begin{lemma}\label{commclasscounts} Let $\mathbf{X}_1$ be a communication class of $K$ whose elements have at least two lower horizontal edges. Then there exist enough $s_x\in S_n$ with $s_x^2\neq id$ to associate a distinct $s_x$ to each $x\in\mathbf{X}_1$ such that $s_x\neq s_y^{-1}$ for any $y\in\mathbf{X}_1$.
\end{lemma}

\begin{proof} The size of a communication class is determined by the number of lower horizontal edges of its elements. Let $\mathbf{X}_i$ be the communication class of an element $x_i$ with $m$ lower horizontal edges. Then a simple counting argument gives:
$$ \vert\mathbf{X}_i\vert =(n-2m)!\prod_{j=0}^{m-1} {n-2j\choose k} .$$

In particular, for $x_i, x_j\in \mathcal{B}_n$, $$\vert \mathbf{X}_i\vert > \vert \mathbf{X}_j\vert\iff \text{ $x_i$ has fewer lower horizontal edges than $x_j$}.$$

Note that if $x_i$ has exactly one lower horizontal edge,
 
 $$|\mathbf{X}_i|=\frac{n!}{2}=\frac{\vert S_n\vert}{2},$$
and so $S_n$ cannot contain enough elements of order greater than $2$ to make the associations required by the lemma, as we need $2\vert\mathbf{X}_i\vert$ elements of order greater than 2.

Now let $x_i$ have exactly two lower horizontal edges. Then for all $x_j\in\mathcal{B}_n$ with at least two lower horizontal edges,

$$\vert\mathbf{X}_j\vert\leq\vert\mathbf{X}_j\vert= (n-4)!{n\choose 2}{n-2\choose 2}=\frac{n!}{8}=\frac{\vert S_n\vert}{8},$$
and so \begin{equation}\label{eqqq}2\vert\mathbf{X}_j\vert\leq \frac{\vert S_n\vert }{2}. \end{equation}

Let ${T}_n$ be the set of elements of $S_n$ of order $2$. Then by Equation \ref{eqqq} we need show $$\frac{\vert S_n\vert }{2}\leq \vert S_n\vert-\vert T_n\vert,$$
in other words, that $\vert T_n\vert \leq \frac{\vert S_n\vert} {2}.$

But
$$ \vert {T}_n\vert = \left\{
     \begin{array}{ll}
     \displaystyle   \sum_{k=1}^{\frac{n}{2}} \frac{n!}{(n-2k)!k!2^k} &  \text{if}\;n\;\text{even}\; ,\\
     &\\
     \displaystyle  \sum_{k=1}^{\frac{n-1}{2}} \frac{n!}{(n-2k)!k!2^k} &  \text{if}\;n\;\text{odd}\; ,\\
       \end{array}
   \right.
$$
and so $\vert T_n\vert < \frac{\vert S_n\vert }{2}$ for $n>4$. As the only communication classes when $n<4$ correspond  to elements with fewer than $2$ lower horizontal edges, this proves the lemma. 
\end{proof}

Finally, for $x_i$ with exactly one lower horizontal edge, while $\mathbf{X}_i$ contains too many elements to make the associations of Lemma \ref{commclasscounts}, note that each $y\in\mathbf{X}_i$ can be viewed as an element, $y'$ of $\mathcal{B}_{n+1}$ by adding a verticle edge to the end of the diagram. Then to analyze $[K]_i$, let $\mathbf{X}_i'=\{y'\mid y\in\mathbf{X}_i\}$ and let $K'$ be the matrix of $K$ with respect to $\mathcal{B}_{n+1}$. Then since $[K']_i=[K]_i$, we can analyze this case by considering $K'$.

\bibliographystyle{alpha}
\bibliography{thesisbib}

\end{document}